\documentclass[leqno,11pt]{scrartcl}
\usepackage[utf8]{inputenc}

\usepackage{amsmath,amssymb,amsthm}
\usepackage{amscd}
\usepackage{setspace} 
\usepackage{enumerate}
\usepackage{array}
\usepackage{tabularx}
\usepackage{booktabs}
\usepackage{url}

\theoremstyle{definition}
\newtheorem{definition}{Definition}
\newtheorem{remark}{Remark}

\newtheorem*{exmps}{Examples}
\theoremstyle{plain}
\newtheorem{theorem}{Theorem}

\newtheorem{lemma}[definition]{Lemma}
\newtheorem{corollary}{Corollary}

\theoremstyle{remark}

\newcommand{\C}{\mathbb{C}}

\renewcommand{\H}{\mathbb{H}}
\newcommand{\K}{\mathbb{K}}
\newcommand{\N}{\mathbb{N}}

\newcommand{\Q}{\mathbb{Q}}
\newcommand{\R}{\mathbb{R}}
\newcommand{\Z}{\mathbb{Z}}
\newcommand{\Acal}{\mathcal{A}}

\newcommand{\Ccal}{\mathcal{C}}
\newcommand{\Gcal}{\mathcal{G}}
\newcommand{\Hcal}{\mathcal{H}}
\newcommand{\Ical}{\mathcal{I}}

\newcommand{\Ncal}{\mathcal{N}}

\newcommand{\diag}{\operatorname{diag}\,}

\newcommand{\SL}{\operatorname{SL}}
\newcommand{\Sp}{\operatorname{Sp}}
\newcommand{\SU}{\operatorname{SU}}
\newcommand{\SO}{\operatorname{SO}}
\newcommand{\im}{\operatorname{Im}}
\newcommand{\re}{\operatorname{Re}}
\newcommand{\oh}{\mathcal{\scriptstyle{O}}}

\let\leq\leqslant
\let\geq\geqslant

\begin{document}

\begin{center}
\begin{huge}
\begin{spacing}{1.0}
\textbf{The maximal discrete extension of the Hermitian modular group}  
\end{spacing}
\end{huge}

\bigskip
by
\bigskip

\begin{large}
\textbf{Aloys Krieg}, \textbf{Martin Raum} and \textbf{Annalena Wernz\footnote{Aloys Krieg, Annalena Wernz,\\ Lehrstuhl A für Mathematik, RWTH Aachen University, D-52056 Aachen, Germany \\ krieg@rwth-aachen.de, annalena.wernz@rwth-aachen.de \\
\hspace*{-1ex}Martin Raum,\\ Chalmers tekniska högskola och Göteborgs Universitet, Institutionen för Matematiska vetenskaper, SE-412 96, 
Göteborg, Sweden \\
martin@raum-brothers.eu \\
\hspace*{-1ex}Martin Raum was partially supported by Vetenskapsr\aa{}det Grant 2015-04139. \\
\hspace*{-1ex}Annalena Wernz was partially supported by Graduiertenkolleg Experimentelle und konstruktive Algebra at RWTH Aachen University.}}
\end{large}
\vspace{0.5cm}\\
\vspace{1cm}
\end{center}
\noindent\textbf{Abstract.}
Let $\Gamma_n(\oh_\K)$ denote the Hermitian modular group of degree $n$ over an imaginary-quadratic number field $\K$. In this paper we determine its maximal discrete extension in $\SU(n,n;\C)$, which coincides with the normalizer of $\Gamma_n(\oh_\K)$. 
The description involves the $n$-torsion subgroup of the ideal class group of $\K$. This group is defined over a particular number field $\widehat{\K}_n$ and we can describe the ramified primes in it. 
In the case $n=2$ we give an explicit description, which involves generalized Atkin-Lehner involutions. Moreover we find a natural characterization of this group in $\SO(2,4)$.
\medskip

\vspace{2ex}

\newpage

\section{Introduction}

The Hermitian modular group of degree~$n$ over an imaginary-quadratic number field~$\K$ was introduced by H.~Braun~\cite{B2}. Modular forms associated with the Hermitian modular group are among the prime examples of automorphic forms for classical groups, illustrated by for instance Mok's endoscopic classification~\cite{mok-2015}. In this work, we determine the maximal discrete extension in $\SU(n,n;\C)$ of the Hermitian modular group, which we call the extended Hermitian modular group. We examine in detail the case of~$n = 2$ and its relation to orthogonal modular groups, and investigate the fields of definition of the extended Hermitian modular group.

Discrete extensions of classical lattices in Lie groups have remarkable properties and applications. Already the case of congruence subgroups~$\Gamma_0(N) \subseteq \SL(2;\Z)$ admits interesting discrete extensions. For example, the Fricke groups~$\Gamma_0^+(N) \subseteq \SL(2;\R)$ played a prominent role in Monstrous Moonshine and its elusive genus-$0$ properties~\cite{borcherds-1992}. Fricke groups~$\Gamma_0^+(N)$ are maximally discrete and are the normalizer of~$\Gamma_0(N)$ in~$\SL(2;\R)$. Moreover, they are generated by Atkin-Lehner involutions, whose entries lie in~$\Q(\sqrt{d})$ for suitable~$d \mathop{\mid} N$.

The symplectic group~$\Sp(n;\Z) \subseteq \Sp(n;\R)$ is the most common case of higher rank groups. It is deceptively simple in our context, since~$\Sp(n;\Z)$ is already maximally discrete (cf.~\cite{Ra}). Nevertheless, there is an analogue of Fricke groups for the level-$N$ paramodular modular groups in~$\Sp(2;\R)$. These groups, in analogy with the case of~$\Gamma_0(N) \subseteq \SL(2;\R)$, are generated by paramodular Atkin-Lehner involutions, whose entries lie in~$\Q(\sqrt{d})$ for suitable~$d \mathop{\mid} N$. The resulting paramodular Fricke groups are of great use in Gritsenko's generalization of the Maa\ss\ lift to the paramodular setting~\cite{gritsenko-1995}, which in turn play a major role in modularity conjectures for abelian surfaces.

While the utility of maximal discrete extensions is perfectly illustrated by such applications, the examples given so far might indicate that their structure is comparatively simple. The present study of Hermitian modular groups shows the opposite: The maximal discrete extension of the Hermitian modular group over~$\K$ reflects properties of the class group of~$\K$. It also yields a new invariant of~$\K$. Specifically, the entries of elements of the maximal discrete extension belong to the ring of integers of an interesting number field $\widehat{\K}_n$ depending on $\K$ and $n$, which is closely related to the Hilbert class field of~$\K$. In Theorem~\ref{Theorem4} we show that the field extension $\widehat{\K}_n\supseteq\Q$ is ramified exactly at the primes dividing~$nd_{\K}$.

In order to state our main theorem, we need some notation. The special unitary group~$\SU(n,n; \C)$ and the Hermitian modular group~$\Gamma_n(\oh_\K)$ are defined in Section~\ref{Preliminaries}. Given a matrix~$L \in \oh_\K^{2n \times 2n}$, we let~$\Ical(L)$ be the ideal generated by the entries of $L$. Moreover, we let $\Ccal_r$ be a cyclic group of order $r$, and
\begin{gather*}
  \Ccal\ell_\K = \{[\Acal];\;\Acal\text{ fractional ideal in }\K\}
\quad\text{and}\quad
  \Ccal\ell_\K[n] := \{[\Acal]\in\Ccal\ell_\K;\;[\Acal]^n = [\oh_\K]\}
\end{gather*}
stand for the ideal class group of $\K$ with class number~$h_\K$ and its $n$-torsion subgroup.

\begin{theorem}\label{Theorem1_neu} 
Let $\K$ be an imaginary-quadratic number field with ring of integers~$\oh_\K$.
\begin{enumerate}
\item[a)]	
The following is a subgroup of\/ $\SU(n,n;\C)$ containing $\Gamma_n(\oh_\K)$:
\begin{gather*}
  \Delta_{n,\K}^\ast
:=
  \left\{
  \tfrac{1}{u} L\in \SU(n,n;\C);\;
  L\in\oh_\K^{2n\times 2n},\,
  0 \ne u\in\C,\,u^n\in\oh_\K,\,
  u^n\oh_\K = \Ical(L)^n
  \right\}
\text{.}
\end{gather*}

\item[b)]
The map
\begin{gather*}
  \Delta_{n,\K}^* \to \Ccal\ell_\K[n],\quad M=\tfrac{1}{u} L\mapsto [\Ical(L)],
\end{gather*}
is a surjective homomophism of the groups. If $d_\K \neq -3,-4$ its kernel is equal to
\begin{gather*}
  \{\varepsilon M;\;\varepsilon\in\C,\, \varepsilon^{2n} = 1,\, M\in \Gamma_n(\oh_\K)\}
\text{.}
\end{gather*}

\item[c)]
The group~$\Delta_{n,\K}^*$ is the maximal discrete extension of\/ $\Gamma_n(\oh_\K)$ in $\SU(n,n;\C)$ and coincides with the normalizer of\/ $\Gamma_n(\oh_\K)$ in $\SU(n,n;\C)$. The factor group~$\Delta_{n,\K}^* / \Gamma_n (\oh_\K)$ is isomorphic to
\begin{gather*}
 \Ccal_n \times \Ccal\ell_\K [n]
\text{,}\qquad
  \text{if }d_\K \neq -3,-4
\text{.}
\end{gather*}
\end{enumerate}
\end{theorem}
We call $\Delta_{n,\K}^\ast$ in Theorem~\ref{Theorem1_neu} the \emph{extended Hermitian modular group} of degree $n$. The proof of Theorem~\ref{Theorem1_neu} is inspired by ideas presented in~\cite{EGM} and will be given in Section~\ref{sec:discrete_subgroups}. Specifically, we build up and employ a normal form theory for elements of~$\Delta_{n,\K}^\ast$. To conclude the maximality of~$\Delta_{n,\K}^\ast$, we crucially employ Corollary~\ref{Corollary1_neu}, which itself rests on a characterization of discrete subgroups of~$\SU(n,n; \K) \subseteq \SU(n,n; \C)$ that extend the Hermitian modular group.

Due to the work of Borcherds there is special interest in the case $n = 2$, which can be viewed as an orthogonal group $\SO(2,4)$. In Theorem~\ref{Theorem3_neu} we derive an explicit isomorphism. We describe the extended Hermitian modular group of degree~$2$ explicitly by means of generalized Atkin-Lehner involutions and show that it admits a natural description in the orthogonal context.

\section{Preliminaries}
\label{Preliminaries}

The special unitary group $\SU(n,n;\C)$ consists of all matrices
\begin{gather*}\tag{1}\label{gl_1}
M=\left(\begin{smallmatrix}
          A & B \\ C & D
        \end{smallmatrix}\right)
  \in SL_{2n}(\C) \;\text{satisfying}\; \overline{M}^{tr} JM = J, \;
  J = \left(\begin{smallmatrix}
   0 & -I \\ I & 0
  \end{smallmatrix}\right), \; 
  I = \left(\begin{smallmatrix}
   1 & & 0 \\ & \ddots & \\ 0 & & 1
  \end{smallmatrix}\right), 
\end{gather*}
where the blocks $A,B,C,D$ are always square matrices.

\begin{lemma}\label{Lemma1}
Given $M= \left(\begin{smallmatrix}
   A & B \\ C & D
  \end{smallmatrix}\right) \in \SU(n,n;\C)$ then
  \[
   \det A, \;\det B, \;\det C,\; \det D \in\R.
  \]
\end{lemma}

\begin{proof}
It suffices to show the result for the A-block, because the other cases are obtained from multiplication with $J$. As $\det A = 0$ is clear, let $\det A \neq 0$. Then we have
\[
 \begin{pmatrix}
  I & 0 \\ -CA^{-1} & I
 \end{pmatrix} M 
 \begin{pmatrix}
  I & -A^{-1}B \\ 0 & I
 \end{pmatrix}  = 
 \begin{pmatrix}
  A & 0 \\ 0 & \overline{A}^ {tr-1}
 \end{pmatrix} 
\]
due to \eqref{gl_1}. Hence
\[
 1 = \det M = (\det A)\big/\,\overline{(\det A)}
\]
yields the claim.
\end{proof}

The group $\SU(n,n;\C)$ acts on the Hermitian half-space (cf.~\cite{B2})
\[
 \H_n:= \{Z\in \C^{n\times n};\; \tfrac{1}{2i}(Z-\overline{Z}^{tr})\;\text{positive definite}\}
\]
via
\[
 Z\mapsto M\langle Z\rangle = (AZ+B)(CZ+D)^{-1}.
\]

Throughout this paper let
\[
 \K = \Q(\sqrt{-m}) \subseteq \C, \;\; m\in \N \;\;\text{squarefree},
\]
be an imaginary-quadratic number field. Its discriminant and ring of integers are 
\[
 d_\K = \begin{cases}
        -m \\ -4m
       \end{cases}
\text{and}\;\;\oh_\K = \Z+ \Z\omega_\K =\begin{cases}
                               \Z+\Z(1+\sqrt{-m})/2, & \text{if}\;\,m\equiv 3\!\!\!\pmod{4},  \\
                               \Z+\Z\sqrt{-m}, & \text{if}\;\, m\equiv 1,2 \!\!\!\pmod{4}.
                              \end{cases}
\]
Denote its unit group by $\oh_\K^*$.

The \textit{Hermitian modular group} of degree $n$ is given by 
\[
 \Gamma_n(\oh_\K):= \SU(n,n;\C)\cap \oh_\K^{2n\times 2n}.
\]
It is well-known that $\Gamma_1(\oh_\K) = SL_2(\Z)$.

\subsection{Reduction to triangular form}

We review two lemmas on discrete subgroups~$\Delta_{n,\K}$ of~$\SU(n,n;\C)$, whose proofs follow closely previous work. In the first lemma we determine the integrality properties of the entries of any~$M \in \Delta_{n,\K}$.

\begin{lemma}\label{Lemma2}
Let $\Delta_{n,\K}$ be a discrete subgroup of $\SU(n,n;\C)$ containing $\Gamma_n(\oh_\K)$ or a subgroup of $\SU(n,n;\C)$, which contains $\Gamma_n(\oh_\K)$ as a normal subgroup. Given $M\in\Delta_{n,\K}$ there exists $u\in \C\backslash\{0\}$ such that
 \[
 uM\in \oh_\K^{2n\times 2n}.
 \]
 Any such $u$ satisfies $\ell = |u|^2 \in \N$.
\end{lemma}

\begin{proof}
 $\Gamma_n(\oh_\K)$ possesses a fundamental domain of finite positive volume in $\H_n$ due to \cite{B2}. 
If $\Delta_{n,\K}$ is discrete, we can proceed in exactly the same way as Ramanathan \cite{Ra} in the proof of his Theorem 1. We have $r=[\Delta_{n,\K}:\Gamma_n(\oh_\K)]< \infty$ and conclude
\[
 (MRM^{-1})^s \in\Gamma_n(\oh_\K) \;\,\text{for all}\;\, R\in\Gamma_n(\oh_\K), \;\, s:=r!\,.
\]
We use this for
\begin{gather*}\tag{2}\label{gl_2}
 R = \begin{pmatrix}
      I & H \\ 0 & I
     \end{pmatrix}, \;
 \begin{pmatrix}
      I & 0 \\ H & I
     \end{pmatrix}, \;\;    
H = \overline{H}^{tr} \in\oh_\K^{n\times n}.
\end{gather*}
If $ R= \left(\begin{smallmatrix}
               A^* & B^* \\ C^* & D^*
              \end{smallmatrix}\right)$, $A^* = (a_{ij})$,
we end up with
\[
s a_{ij} \overline{a}_{k\ell} \in\oh_\K \;\; \text{for all}\;\;i,j,k,\ell = 1,\ldots,n.
\]
The same holds for $B^*$, $C^*$, $D^*$. Thus the existence of $u\in \C\backslash\{0\}$ satisfying
\[
 uM = L\in \oh_\K^{2n\times 2n}
\]
follows. The identity
\[
 |u|^2 J = (\overline{uM})^{tr} J (uM) = \overline{L}^{tr} J L
\]
yields $|u|^2 \in \oh_\K\cap \R = \Z$, because the elements on the right hand side are integral. If $\Gamma_n(\oh_\K)$ is normal in $\Delta_{n,\K}$, we can take the same arguments with $s=1$.
\end{proof}

The second lemma in this section is a result of Ensenbach~\cite{En} (cf.~\cite{Rm}), which allows us to simplify the shape of $M \in \Delta_{n,\K}$.

\begin{lemma}\label{Lemma3}
Let $\Delta_{n,\K}$ be a discrete subgroup of $\SU(n,n;\C)$ containing $\Gamma_n(\oh_\K)$ or a subgroup of $\SU(n,n;\C)$, which contains $\Gamma_n(\oh_\K)$ as a normal subgroup. Given $M\in \Delta_{n,\K}$ there exists an $R\in\Gamma_n(\oh_\K)$ such that
 \[
  RM = \begin{pmatrix}
        A^* & B^* \\ 0 & D^*
       \end{pmatrix}, \;\; A^*\in SL_n(\C).
 \]
\end{lemma}

\begin{proof}
Apply Lemma \ref{Lemma2} and let $L= uM = \left(\begin{smallmatrix}
             A & B \\ C & D
            \end{smallmatrix}\right) \in \oh_\K^{2n\times 2n}$. 
Then $(\overline{C}^{tr}, \overline{A}^{tr})$ is a Hermitian pair in the sense of Braun \cite{B}. Due to \cite {B}, Theorem \ref{Theorem3_neu}, there exists a coprime pair in this class, which can be completed to a matrix in $\Gamma_n(\oh_\K)$. Multiplying by its inverse yields the shape of $RM$. If $\Delta_{n,\K}$ is discrete, we obtain $|\det A^*| = 1$ as the index $[\Delta_{n,\K}:\Gamma_n(\oh_\K)]$ is finite. If $\Gamma_n(\oh_\K)$ is nomal in $\Delta_{n,\K}$ then \eqref{gl_2} for $H=I$ yields $|\det A^*|^2$, $|\det D^*|^2 \in \Z$, hence $|\det A^*| = 1$. In view of Lemma \ref{Lemma1} we may multiply by 
\[
 \begin{pmatrix}
			    P & 0 \\ 0 & P
			    \end{pmatrix},\;\; P = \diag (1,\ldots,1,-1), 
\]
in order to obtain $\det A^* = 1$.          
\end{proof}

\section{Consequences for discrete subgroups}
\label{sec:discrete_subgroups}

We start our investigation of discrete subgroups of~$\SU(n,n; \C)$, with the special case of subgroups of~$\SU(n,n;\K)$.

\begin{lemma}\label{Lemma5_neu} 
Let $\Delta_{n,\K}$ be a discrete subgroup of $\SU(n,n;\K)$ containing $\Gamma_n(\oh_\K)$. Then 
\[
 \Delta_{n,\K} = \Gamma_n(\oh_\K).
\]
\end{lemma}

\begin{proof}
We assume that there exists an $M\in\Delta_{n,\K}$ with an entry in $\K\backslash \oh_\K$. Then there is a prime ideal $\wp\subseteq \oh_\K$ and an entry $x$ of $M=(m_{ij})$ such that the exponent $e(x,\wp)$ of the prime ideal decomposition of $\oh_\K x$ is $< 0$. Choose $M= \left(\begin{smallmatrix}
	A & B \\ 0 & D
	\end{smallmatrix}\right)$, $A=(a_{ij})$, as in Lemma \ref{Lemma3} with $C=0$ and assume without restriction that
\[
 e(a_{11},\wp) \leq \min\{e(m_{ij},\wp),\;1\leq i,j \leq 2n\}.
\]
After multiplying with matrices
\[
 \begin{pmatrix}
  \overline{U}^{tr} & 0 \\  0 & U^{-1}
 \end{pmatrix}, \;
  U = \begin{pmatrix}
  1 & 0 \\  g & I
 \end{pmatrix}, \;
 \begin{pmatrix}
  1 & \overline{g}^{tr} \\  0 & I
 \end{pmatrix}, \;\; g\in\oh_\K^{n-1},
\]
we may assume 
\[
 e(a_{11},\wp)<\min\{e(a_{j1},\wp),\,e(a_{1j},\wp),\, j=2,\ldots,n\}.
\]
Considering $A^r= (a_{ij}^*)$, $r\in \N$, we get
\begin{align*}
 & e(a_{11}^*,\wp) = r \, e(a_{11},\wp) \leq e(a_{ij}^*,\wp), \;\, 1\leq i,j \leq n\},  \\
 & e(a_{11}^*,\wp) < \min\{ e(a_{j1}^*,\wp),\, e(a_{1j}^*,\wp), \;\, j= 2,\ldots,n\}.
\end{align*}
Since the multiplication by unimodular matrices does not change the ideal generated by the entries of $M^r$, the cosets 
\[
 \Gamma_n(\oh_\K) M^r,\;\, r\in\N_0,
\]
are mutually distinct. This contradicts
\[
 [\Delta_{n,\K}:\Gamma_n(\oh_\K)] < \infty.
\]
\vspace*{-8ex}\\
\end{proof}
\vspace{0,5ex}

An immediate application is

\begin{corollary}\label{Corollary1_neu} 
Let $\Delta_{n,\K}$ be a discrete subgroup of $\SU(n,n;\C)$ containing $\Gamma_n(\oh_\K)$. Then $\Delta_{n,\K}$ is contained in the normalizer of $\Gamma_n(\oh_\K)$ in $\SU(n,n;\C)$.
\end{corollary}

\begin{proof}
Let $M\in\Delta_{n,\K}$, $u\in\C\backslash\{0\}$, $M=\frac{1}{u}L$, $L\in\oh_\K^{2n\times 2n}$ according to Lemma \ref{Lemma2}. Then
\[
 M\Gamma_n(\oh_\K) M^{-1} = L \Gamma_n(\oh_\K) L^{-1} \subseteq \Delta_{n,\K} \cap \SU(n,n;\K).
\]
Hence we can apply Lemma \ref{Lemma5_neu} to the group generated by $\Gamma_n(\oh_\K)$ and $M\Gamma_n(\oh_\K) M^{-1}$. The result is
\[
 M\Gamma_n(\oh_\K) M^{-1} = \Gamma_n(\oh_\K).
\]
\vspace*{-7ex}\\
\end{proof}
\vspace{0,5ex}

Now let $M=\frac{1}{u}L$, $L\in\oh_\K^{2n\times 2n}$, belong to the normalizer of $\Gamma_n(\oh_\K)$ in $\SU(n,n;\C)$, $\ell = u\overline{u}\in\N$ and recall that $\Ical(L)$ stands for the ideal generated by the entries of $L$. If $\Ncal$ denotes the reduced norm of an ideal in $\oh_\K$, we obtain

\begin{lemma}\label{Lemma6_neu} 
$\Ical(L)$ is an invariant of the double coset with respect to $\Gamma_n(\oh_\K)$. One has  \\[0,5ex]
a) $\ell = \Ncal(\Ical(L))$, $\ell \oh_\K = \Ical(L) \cdot \overline{\Ical(L)}$,  \\ [0,5ex]
b) $u^n \in \oh_\K$, $u^n \oh_\K = \Ical(L)^n$.
\end{lemma}

\begin{proof}
$\Ical(L)$ does not change, if we multiply by unimodular matrices. Using Lemma \ref{Lemma3} with $A^* = \frac{1}{u} A$ we get 
\[
 u^n=\det A \in\oh_\K.
\]
Then $A\overline{D}^{tr} = \ell I$ yields $\ell\in \Ical(L)\cdot \overline{\Ical(L)}$. As the proof of Lemma \ref{Lemma2} shows
\[
 x\overline{y} \in\ell\oh_{\K} \quad \text{for all entries}\quad x,y \;\,\text{of}\;\, L,
\]
we obtain
\[
 \Ical(L)\cdot \overline{\Ical(L)} \subseteq \ell \oh_\K
\]
and therefore equality. Computing the reduced norm $\Ncal$ we get
\[
 \ell^2 = \Ncal(\ell \oh_\K) = \Ncal(\Ical(L)\cdot \overline{\Ical(L)}) = \Ncal(\Ical(L))^2.
\]
Computing the determinant we obtain
\[
 u^n= \det A \in \Ical(L)^n, \;\;\text{i.e.}\;\; u^n \oh_\K \subseteq \Ical(L)^n
\]
hence
\[
 \ell^n = |\det A|^2 = \Ncal(u^n \oh_\K) \geq \Ncal(\Ical(L))^n = \ell^n.
\]
Therefore $u^n \oh_\K = \Ical(L)^n$ follows.
\end{proof}

We are now in position to establish our main theorem.

\begin{proof}[Proof of Theorem~\ref{Theorem1_neu}]
a) Given $M = \frac{1}{u} L$, $M'=\frac{1}{u'} L'\in\Delta_{n,\K}^*$ we have
\[
 (uu')^n\oh_\K \subseteq \Ical(LL')^n\subseteq \Ical(L)^n\cdot \Ical(L')^n = u^n\oh_\K \cdot u'^n\oh_\K,
\]
hence $MM'\in\Delta_{n,\K}^*$. The result follows from
\[
 M^{-1} = \begin{pmatrix}
           \overline{D}^{tr} & -\overline{B}^{tr} \\ -\overline{C}^{tr} & \overline{A}^{tr}
          \end{pmatrix}. 
\]
b) At first $M=\frac{1}{u} L = \frac{1}{u'} L'$ with integral $L,L'$ implies
\[
 \tfrac{u'}{u} I = L'L^{-1} \in GL_{2n} (\K),
\]
hence
\[
 \tfrac{u'}{u} \in \K \;\;\text{and}\;\; \tfrac{u'}{u} \oh_{\K} \cdot \Ical(L) = \Ical(L').
\]
Thus $[\Ical(L)]\in \Ccal\ell_\K[n]$ is well-defined by $M$. Hence the map is a homomophism of the groups due to Lemma \ref{Lemma6_neu} and a). If $d_\K\neq -3,-4$ we have
\[
u^n \oh_\K = v^n\oh_\K \;\Leftrightarrow\; u^n= \pm v^n \;\Leftrightarrow\; u=\varepsilon v, \;\; \varepsilon^{2n}=1.
\]
Given $[\Acal]\in\Ical_\K[n]$ we may assume $\Acal\subseteq \oh_\K$. According to \cite{En2}, Satz 2.1, resp. \cite{Rm2} there is a matrix $A\in \oh_\K^{n\times n}$ with elementary divisors $\Acal,\ldots,\Acal$ and $\oh_\K \det A = \Acal^n$, i.e. 
\begin{gather*}\tag{3}\label{gl_3neu_neu}
 A\in \Acal^{n\times n}, \; \oh_\K \det A = \Acal^n, \; u:=\sqrt[n]{\det A} \in\C,\; \ell = u\overline{u} = \Ncal(\Acal).
\end{gather*}
Then we have $AGA^{-1}\in\oh_\K^{n\times n}$ for all $G\in\oh_\K^{n\times n}$ as well as $\frac{1}{\ell} A\overline{A}^{tr} \in SL_n(\oh_\K)$. Thus
\[
 M = \tfrac{1}{u}\begin{pmatrix}
      A & 0 \\ 0 & D
     \end{pmatrix}\in\Delta_{n,\K}^*, \;\; D= \ell \overline{A}^{tr-1}
\]
follows. As this $M$ is mapped onto $[\Acal]$, the map is surjective. \\[0,5ex]
c) Observe that the image consists of at most $h_\K$ elements. Thus 
\[
[\Delta_{n,\K}^*:\Gamma_n(\oh_\K)]< \infty 
\]
follows. Hence $\Delta_{n,\K}^*$ is discrete. Thus it coincides with the normalizer of $\Gamma_n(\oh_\K)$ according to Corollary 
\ref{Corollary1_neu}. The result on the factor group follows from b).
\end{proof}

It is a consequence of the proof that representatives of the $\Gamma_n(\oh_\K)$-cosets in $\Delta_{n,\K}^*$ can be given in the form
\begin{gather*}\tag{4}\label{gl_3neu}
 \begin{pmatrix}
  \overline{A}^{tr} & 0 \\ 0 & A^{-1}
 \end{pmatrix},  \;\;
A=\tfrac{1}{u}L\in SL_n(\C),\;u\in\C,\; L\in\oh_\K^{n\times n},\;\; \Ical(L)^n=\oh_\K \det L .
\end{gather*}

We formulate a special case and include the cases $d_\K = -3,-4$ from Theorem \ref{Theorem1_neu} in 

\begin{corollary}\label{Corollary2_neu} 
Let $\K$ be an imaginary-quadratic number field such that the class number $h_\K$ is coprime to $n$. Then the extended Hermitian modular group is equal to 
\[
 \bigl\{\diag(\varepsilon,\ldots,\varepsilon,\varepsilon \delta,\varepsilon,\ldots,\varepsilon,\varepsilon \delta)M; \; M\in\Gamma_n(\oh_\K),\; \varepsilon\in\C,\; \delta \in \oh_\K^*, \; \varepsilon^{2n} \delta^2 = 1\bigr\}.
\]
\end{corollary}

If $n=1$ our result reproves the fact that $SL_2(\Z)=\Gamma_1(\oh_\K)= \Delta_{1,\K}^*$ coincides with its normalizer and maximal discrete extension in $SL_2(\R)$.

The same arguments as in \eqref{gl_3neu} can be used in order to obtain

\begin{corollary}\label{Corollary3} 
Let $\K$ be an imaginary-quadratic number field. Then the maximal discrete extension of $SL_n(\oh_\K)$ in $SL_n(\C)$ coincides with the normalizer of $SL_n(\oh_\K)$ in $SL_n(\C)$ and is given by
\[
 \bigl\{\tfrac{1}{u} L\in SL_n(\C);\; L\in \oh_\K^{n\times n},\,0\neq u\in\C,\, \oh_\K \det L = \Ical(L)^n\bigr\}.
\]
\end{corollary}

The explicit description of this group in the case $n=2$ is contained in \cite{KRW}.

\begin{remark}\label{Remark1} 
a) The result for $\SU(n,n;\K)$ differs from the corresponding result for $Sp(n;\K)$ in \cite{Ra}, Theorem 2, as the index $[\Delta_{n,\K}^*:\Gamma_n(\oh_\K)]$ is not independent of $n$. Moreover the factor group does not only contain elements of order $1$ or $2$, if $n> 2$. \\
b) The entries of the coset
 \[
  \begin{pmatrix}
   \overline{A}^{tr} & 0 \\ 0 & A^{-1}
  \end{pmatrix} \Gamma_n(\oh_\K)
 \]
in \eqref{gl_3neu} are algebraic integers of the algebraic number field $\Q(\sqrt{-m},u)$ of degree $\leq 2n$, because $(a/u)^n\in \oh_\K$ for any entry $a$ of $A$. \\
c) It follows from \eqref{gl_3neu} that 
\[
 M^n \in \Gamma_n(\oh_\K) \;\; \text{for all}\;\; M\in \Delta_{n,\K}^*.
\]
\end{remark}

\section{Orthongonal groups}

We have a closer look at the Hermitian modular group for $n=2$. We also revisit a family of isomorphisms between~$\SU(2,2;\C) \slash \{\pm I\}$ and the connected component~$\SO_0(2,4;\R)$ of the special orthogonal group~$\SO(2,4;\R)$. We determine isomorphisms that are compatible with the discrete subgrous~$\Gamma_n(\oh_\K)$ and thus allow us to shed a different light on our main theorem.

Let $d\in\N$ be a squarefree divisor of $d_\K$. Then
\[
 \Acal_d = \oh_\K d + \oh_\K(m+\sqrt{-m})\subseteq \oh_\K
\]
is the unique ideal of reduced norm $d$ in $\oh_\K$. If $e |d_\K$ is squarefree and coprime to $d$ the prime ideal decomposition 
yields 
\[
 \Acal_d^2 = d\oh_\K, \quad \Acal_d \cdot \Acal_e = \Acal_{de}.
\]
Following \cite{KRW} we determine $u,v\in\Z$ such that
\[
 ud-v(m^2+m)/d = 1,
\]
i.e.
\begin{gather*}\tag{5}\label{gl_4}
 V_d:= \frac{1}{\sqrt{d}} \left(\begin{smallmatrix}
                           ud & v(m+\sqrt{-m}) \\ m-\sqrt{-m} & d
                          \end{smallmatrix}\right)
 \in SL_2(\C), \;\; W_d:= \left(\begin{smallmatrix}
                           \overline{V}_d^{tr} & 0  \\  0 & V_d^{-1}
                          \end{smallmatrix}\right) \in \SU(2,2;\C).
\end{gather*}
Then we get
\[
 \Gamma_2(\oh_\K) W_d = W_d \Gamma_2(\oh_\K) = \biggl\{\frac{1}{\sqrt{d}} M\in \SU(2,2;\C);\;M\in\Acal_d^{4\times 4}\biggr\}.
\]

\begin{theorem}\label{Theorem2}
 For an imaginary-quadratic number field $\K$, the extended Hermitian modular group of degree $2$ is given by
 \[
  \Delta_{2,\K}^* = \bigcup_{d|d_\K,d \;\square-\text{free}} \Gamma_2(\oh_\K) W_d.
 \]
It contains $\Gamma_2(\oh_\K)$ as a normal subgroup of index $2^{\nu}$, 
\[
\nu:=\sharp\{p;p\;\text{prime},\; p | d_\K\},
\]
and the factor group satisfies 
\[
 \Delta_{2,\K}^*/\Gamma_2(\oh_\K) \cong \Ccal_2^{\nu}.
\]

\end{theorem}
%
\begin{proof}
The result follows from Theorem \ref{Theorem1_neu}, Corollary \ref{Corollary2_neu} and \cite{KRW}. But we also give a direct proof. Let
\[
 M=\tfrac{1}{u} L = \tfrac{1}{u} \begin{pmatrix}
                                  A^* & B^* \\ 0 & D^*
                                 \end{pmatrix}
\in \Delta_{2,\K}^*
\]
according to Theorem \ref{Theorem1_neu} and Lemma \ref{Lemma3}. Then $\Ncal(A^*) = \ell = u\overline{u} \in\N$ and $|\det A^*|^2 = \ell^2$. If we consider $(\ell + \overline{\det A^*})M$ instead of $M$, if necessary, we may assume $\det A^* \in\N$. Cancelling integers we may even assume $\frac{1}{r} A^*\not\in \oh_\K^{2\times 2}$ for $r>1$, hence
\[
 M= \frac{1}{\sqrt{\ell}}\, L.
\]
In view of $\ell \oh_\K = \Ical(L)^2$ due to Theorem \ref{Theorem1_neu}, we conclude that $\ell$ is a squarefree divisor of $d_\K$ from the prime ideal decomposition in $\oh_\K$.
\end{proof}

We consider particular examples in 

\begin{remark}\label{Remark2}
a) One has
\[
 iI \in \Gamma_2(\oh_\K) W_m,
\]
hence
\[
 \Delta_{2,\K}^* = \Gamma_2(\oh_\K) \cup (iI)\Gamma_2(\oh_\K),
\]
if $d_\K\neq -4$ is a prime discriminant. If $d_\K = -4$ one has
\[
 \left(\tfrac{1+i}{\sqrt{2}}\right) \begin{pmatrix}
                                    U & 0 \\ 0 & U
                                   \end{pmatrix}
\in\Gamma_2(\oh_\K)W_{2}, \;\; U = \begin{pmatrix}
                                     1 & 0 \\ 0 & i
                                    \end{pmatrix},
\]
and therefore 
\[
 \Delta_{2,\K}^* = \Gamma_2(\oh_\K) \cup\left(\tfrac{1+i}{\sqrt{2}}\right) \begin{pmatrix}
                                                             U & 0 \\ 0 & U
                                                            \end{pmatrix}
\Gamma_2(\oh_\K). 
\]
This group is isomorphic to $U(2,2;\oh_\K)$ and already appeared, when the attached graded ring of Hermitian modular forms was determined (cf. \cite{F2}).  \\[0.5ex]
b) Theorem \ref{Theorem1_neu} and Theorem \ref{Theorem2} reprove a result of Hecke that 
\[
 \Ccal\ell_\K [2] \cong \Ccal_2^{\nu-1}.
\]
c) If $\Ccal\ell_\K = \Ccal\ell_\K [2]$ we conclude that
\[
 \Delta_{n,\K}^* = 
 \begin{cases}
\bigcup\limits_{\varepsilon\in\C, \,\varepsilon^{2n} = 1} \varepsilon I\cdot \Gamma_2(\oh_\K), & \text{if}\;n\;\text{is odd}, \\
 \bigcup\limits_{\substack{\varepsilon\in\C,\, \varepsilon^{2n} = 1\\ d|d_\K,\,d\;\square-\text{free}}} \varepsilon I\cdot\widehat{W}_d  \Gamma_n(\oh_\K), & \text{if}\;n\;\text{is even},                  
\end{cases}
\]
where 
\[
 \widehat{W}_d = \diag (\overline{V}_d^{tr},\ldots,\overline{V}_d^{tr}, V_d^{-1},\ldots,V_d^{-1}) \in \SU(n,n;\C).
\]
\end{remark}
\vspace{1.5ex}

Moreover there is another group involved in this context, namely
\[
 \Gcal_\K:=\biggl\{\frac{1}{\sqrt{k}} M \in \SU(2,2;\C);\;k\in \N,\,M\in \oh_\K^{4\times 4}\biggr\}\supseteq \SU(2,2;\K),
\]
which naturally appears in the associated Hecke theory (cf. \cite{En}, \cite{Rm}). \\[-0,5ex]

Now we consider the orthogonal setting. Let
\[
 S=\begin{pmatrix}
    2 & 2\re(\omega_\K) \\ 2 \re(\omega_\K) & 2 \, \omega_\K\overline{\omega}_\K
   \end{pmatrix}, \;
S_0 = \begin{pmatrix}
    0 & 0 & 1 \\ 0 & -S & 0 \\ 1 & 0 & 0
   \end{pmatrix}, \;
S_1 = \begin{pmatrix}
    0 & 0 & 1 \\ 0 & S_0 & 0 \\ 1 & 0 & 0
   \end{pmatrix}.
\]
Then $S_1$ is of signature $(2,4)$. We define the attached special orthogonal group by
\[
 \SO(S_1;\R):=\{M\in SL_6(\R);\;S_1[M]:= M_1^{tr} S_1 M = S_1\}
\]
and denote by $\SO_0(S_1;\R)$ the connected component of the identity matrix $I$. Let
\[
 \Sigma_\K:=\{M\in \SO_0(S_1;\R);\;M\in I+\Z^{6\times 6} S_1\}
\]
stand for \textit{discriminant kernel} and $\Sigma_\K^*:= \SO_0(S_1;\Z)$ for the full group of integral matrices in $\SO_0(S_1;\R)$.

The group $\SO_0(S_1;\R)$ acts on the orthogonal half-space (cf. \cite{G4})
\[
 \Hcal_\K:=\{z=x+iy=(\tau_1,z,w,\tau_2)^{tr} \in \C^4,\, \im \tau_1>0,\, y^{tr}S_0 y > 0\}
\]
via 
\[
 z\mapsto \widetilde{M}\langle z\rangle:=\frac{1}{\widetilde{M}\{z\}}\left(-\tfrac{1}{2} S_0[z]b+Kz+c\right),
\]
where
\[
 \widetilde{M} = \begin{pmatrix}
                  \alpha & a^{tr} S_0 & \beta \\ b & K & c \\ \gamma & d^{tr} S_0 & \delta
                 \end{pmatrix}, \;
\widetilde{M}\{z\} = -\tfrac{\gamma}{2} S_0 [z] + d^{tr} S_0 z+\delta.
\]
We consider the bijection between the (complexified) Hermitian $2\times 2$ matrices and $\R^4$ (resp. $\C^4$)
\[
 \varphi:\begin{pmatrix}
       \alpha & \beta+\gamma\omega_\K \\ \beta+\gamma \overline{\omega}_\K & \delta
      \end{pmatrix}
\mapsto (\alpha, \beta, \gamma, \delta)^{tr},
\]
which satisfies
\[
 \varphi(\H_2) = \Hcal_\K.
\]
There is an isomorphism between $\SU(2,2;\C)/\{\pm I\}$ and $\SO_0(S_1;\R)$, where $\pm M\mapsto \widetilde{M}$, given by
\begin{gather*}
 \varphi(M\langle Z\rangle) = \widetilde{M}\langle \varphi (Z)\rangle, \;\det (CZ+D) = \widetilde{M}\{\varphi(Z)\}\;\,\text{for all}\;\, Z\in\H_2. \tag{6}\label{gl_5} 
\end{gather*}
Just as in \cite{GKr} we obtain an explicit version, if we use the abbreviation
\[
 \begin{pmatrix}
  \alpha & \beta \\ \gamma & \delta
 \end{pmatrix}^{\sharp} 
= \begin{pmatrix}
  \delta & -\beta \\ -\gamma & \alpha
 \end{pmatrix}
\]
for the adjoint matrix, via $\phi(Z) = z$
\[
 \det(CZ+D) = \widetilde{M}\{z\}, \; (AZ+B)(CZ+D)^{\sharp} = -\tfrac{1}{2} S_0[z]b + Kz+c,
\]
where
\begin{gather*}
\left\{\;\;\begin{matrix}
 \alpha = \det A,\; \beta=-\det B, \; \gamma=-\det C, \; \delta = \det D ,\\
 a= -\varphi(A^{\sharp} B), \; b= -\varphi(AC^{\sharp}), \; c= \varphi(BD^{\sharp}), \; d= \varphi(C^{\sharp}D), \tag{7}\label{gl_6} \\ 
 Kz=\varphi(AZD^{\sharp} + BZ^{\sharp}C^{\sharp}).
 \end{matrix}\right.
\end{gather*}

\begin{theorem}\label{Theorem3_neu}
 Let $\K$ be an imaginary-quadratic number field. Then the map
 \[
  \phi: \SU(2,2;\C) \to \SO_0(S_1;\R), \quad M\mapsto \widetilde{M},
 \]
given by \eqref{gl_5} and \eqref{gl_6} is a surjective homomorphism of the groups with kernel $\{\pm I\}$. It satisfies
\begin{align*}
 \phi(\Gcal_\K) & = \SO_0(S_1;\Q),  \\
 \phi(\Delta_{2,\K}^*) & = \Sigma_\K^*,  \\
 \phi(\Gamma_2(\oh_\K)) & = \Sigma_\K.
\end{align*}
\end{theorem}

\begin{proof}
 The groups on the left hand side are generated by the matrices
 \[
  J,\; \begin{pmatrix}
        I & H \\ 0 & I
       \end{pmatrix}, \; H = \overline{H}^{tr}, \; 
\begin{pmatrix}
        \overline{U}^{tr} & 0 \\ 0 & U^{-1}
       \end{pmatrix}, \; \det U\in \R\backslash\{0\},
 \]
with coefficients in the appropriate set. This is clear for $\SU(2,2;\C)$ (cf. \cite{K3}), with a slight adaption also for $\Gcal_\K$, for $\Gamma_2(\oh_\K)$ due to \cite{Kl}, Theorem 3, and for $\Delta_{2,\K}^*$ by virtue of Theorem \ref{Theorem2}. Hence the image is contained in the right hand side in any case. \\
On the other hand it follows from \cite{K1} that the groups on the right hand side are generated by matrices of the type
\begin{align*}
 & \widetilde{J} = \begin{pmatrix}
 0 & 0 & P  \\ 0 & I & 0 \\ P & 0 & 0
 \end{pmatrix},\;
 P = \begin{pmatrix}
  0 & -1 \\ -1 & 0
 \end{pmatrix}, \;
 \begin{pmatrix}
 1 & \lambda^{tr}S_0 & -\tfrac{1}{2}S_0[\lambda]  \\ 0 & I & \lambda \\ 0 & 0 & 1
 \end{pmatrix},\;
\begin{pmatrix}
 1 & 0 & 0  \\ 0 & K & 0 \\ 0 & 0 & 1
 \end{pmatrix},\\[1ex]
 & K\in \SO_0(1,3).
\end{align*}
They appear as images of elements in the groups on the left hand side due to \eqref{gl_6}, if one invokes \cite{KRW} for the last type of matrices. As the kernel of $\phi$ is clearly $\{\pm I\}$, the claim follows.
\end{proof}

\begin{remark}\label{Remark3} 
a) By Theorem \ref{Theorem1_neu} and \ref{Theorem3_neu} clearly $\Sigma_\K^*$ is the normalizer and the maximal discrete extension of $\Sigma_\K$ in $\SO_0(S_1;\R)$.  \\
b) The result on the normalizer for $n=2$ is contained in \cite{We} with a completely different proof. Applications to Hermitian modular forms are described in \cite{We2}. It is clear from \eqref{gl_3neu} that the Siegel-Eisenstein series for $\Gamma_n(\oh_\K)$ (cf. \cite{B}) is a modular form with respect to the extended Hermitian modular group $\Delta_{n,\K}^*$.\\
c) Köhler \cite{Koe2} showed that the paramodular group of degree $2$ and squarefree level $N$ can be embedded into the Hermitian modular group $\Gamma_2(\oh_\K)$, if $N=u\overline{u}$ for some $u\in\oh_\K$. The normalizer of the paramodular group (cf.~\cite{Koe}) can be embedded into $\Delta_{2,\K}^*$, whenever $m=N$.\\
d) Theorem \ref{Theorem3_neu} illustrates that the exceptional isogeny between the real Lie groups $\SU(2,2;\C)$ and $\SO_0(S_1;\R)$ in general does not decend to an isogeny of algebraic groups (cf. \cite{Gar}). Indeed, when constructing the underlying isomorphism of the special unitary and the spin group, one typically employs a diagonal quadratic from of signature $(2,4)$ as opposed to the one that is associated with $S_1$.
\end{remark}

\section{Field extensions associated with discrete groups}

There is no analogue of the isomorphism between $\SU(2,2;\C) \slash \{ \pm I \}$ and $\SO_{0}(S_1;\R)$ for general~$n$,  which we utilized when inspecting the special case~$n=2$. Let~$\mu \subseteq \C$ be the set of~$2n$-th roots of unity. Theorem~1 connects the quotient group~$\Delta_{n,\mathbb{K}}^\ast \slash \mu \Gamma_n(\oh_{\mathbb{K}})$ with the $n$-torsion subgroup of the ideal class group of~$\mathbb{K}$ if $d_\K\neq -3,-4$. Class groups of imaginary-quadratic number fields including their~$n$-torsion are elusive, only conjecturally governed by the Cohen-Lehnstra heuristic~\cite{CoLen} and its recent global refinement~\cite{HoJoKuMcLPe}. For this reason a detailed description of~$\Delta_{n,\mathbb{K}}^\ast$ for general~$n$ is out of reach. We next narrow down the field of definition of~$\Delta_{n,\mathbb{K}}^\ast$ and provide an algorithm to compute representatives of~$\Delta_{n,\mathbb{K}}^\ast \slash \mu \Gamma_n(\oh_{\mathbb{K}})$ for given~$\mathbb{K}$ and~$n$.

The field of definition of~$\Delta_{n,\mathbb{K}}^\ast$ is
\begin{gather*}\tag{8}\label{gl_7}
 \widehat{\K}_n =  \widehat{\K}\big( \Delta_{n,\mathbb{K}}^\ast \big)
\;:=\;
  \Q\big(
  m_{i,j} \,;\,
  M = (m_{ij})\in \Delta_{n,\mathbb{K}}^\ast,\,
  1 \leqq i,j \leq 2 n
  \big)
\subseteq \C .
\end{gather*}
Clearly Remark \ref{Remark1}\,b) yields
\[
 \Delta_{n,\K}^* \subseteq \oh_{\widehat{\K}_n}^{2n\times 2n}.
\]
Obviously $\K\subseteq \widehat{\K}_n$ holds for $n\geq 2$. We quote some examples from our considerations above. Therefore let $\zeta_r = e^{2\pi i/r}$ stand for a primitive $r$-th root of unity.

\begin{exmps}
a) Clearly $\widehat{\K}_1 = \Q$ holds.  \\
b) If $n= 2$ Theorem \ref{Theorem2} yields
\[
 \widehat{\K}_2 = \Q(\zeta_4, \sqrt{p};\;p\;\text{prime},\,p|d_\K).
\]
c) One has for $n\geq 2$
\[
 \widehat{\K}_n = 
\begin{cases}
 \K[\zeta_{4n}], & \text{if}\;d_\K = -4, \\
  \K[\zeta_{6n}], & \text{if}\;d_\K = -3.
\end{cases}
\]
If $h_\K$ is coprime to $n$ and $n\geq 2$, $d_\K\neq -3,-4$, then
\[
 \widehat{\K}_n = \K[\zeta_{2n}]
\]
holds.  \\[0.5ex]
d) If $\Ccal\ell_\K = \Ccal\ell_\K[2]$ and $n\geq 2$, $d_\K\neq -3,-4$, we have
\[
 \widehat{\K}_n = 
 \begin{cases}
  \K[\zeta_{2n}], & \text{if}\;n\;\text{is odd},  \\
  \K[\zeta_{2n}, \sqrt{p},\;p\;\text{prime},\;p|d_{\K}], & \text{if}\;n\;\text{is even}.
 \end{cases}
\]
The examples show that $\widehat{\K}_n$ is an interesting number field. We want to generalize the fact that only ramified primes appear in the definition of $\widehat{\K}_2$ to the case of arbitrary $n$.
\end{exmps}

\begin{theorem}\label{Theorem4}
Let $\K$ be an imaginary-quadratic number field. Then the field extension 
\[
\widehat{\K}_n = \widehat{\K}( \Delta_{n,\mathbb{K}}^\ast ) \supseteq \Q 
\]
is ramified exactly at the primes dividing~$n d_\K$.
\end{theorem}

\begin{proof}
Observe that~$\K$ and~$\mathbb{Q}(\mu)$ are ramified at~$d_\K$ and~$n$. Since~$\widehat{K}_n$ contains both~$\K$ and~$\mathbb{Q}(\mu)$, we conclude that~$\widehat{K}_n$ is ramified at~$n d_\K$.

In the remainder of the proof, we show that~$\widehat{\K}_n$ is unramified outside of~$n d_{\K}$. Recall that the composite of two field extension~$F \supseteq K$ and~$F' \supseteq K$ is unramified at a prime ideal~$\mathfrak{p} \subseteq \oh_K$ if and only if both~$F$ and~$F'$ are so. 

Let $n\geq 2$, as $\widehat{\K}_1 = \Q$ is trivial. Moreover $d_{\K}=-3$ and $d_{\K}= -4$ are clear due to the Example. Therefore let $d_{\K} \neq -3,-4$. 
We have $\mathbb{K} \subseteq \mathbb{K}( \Delta_{n,\mathbb{K}}^\ast )$ and the matrix entries of elements of~$\Gamma_n(\oh_{\mathbb{K}})$ lie in~$\mathbb{K}$. Observe that~$\mathbb{K}$ is unramified outside of~$d_{\mathbb{K}}$. The cyclotomic field~$\mathbb{Q}(\mu)$ is unramified outside of~$n d_\K$. Therefore, we can focus on the ramification of fields generated by representatives of~$\Delta_{n,\mathbb{K}}^\ast \slash \mu \Gamma_n(\oh_{\mathbb{K}})$.

Since $\Delta_{n,\mathbb{K}}^\ast \slash \mu\Gamma_n(\oh_{\mathbb{K}}) \cong \Ccal\ell_\K[n]$ is finite, we conclude that it suffices to show the analogue of Theorem~\ref{Theorem4} for individual elements~$M$ of~$\Delta_{n,\mathbb{K}}^\ast$. This allows us to employ Theorem~1~b) and further restrict to fields generated by the entries of the matrices~$\frac{1}{u} \left(\begin{smallmatrix} A & 0 \\ 0 & D \end{smallmatrix}\right) \in \Delta_{n,\mathbb{K}}^\ast$ in \eqref{gl_3neu_neu}.

Now, recall from, for example, Section~3.1 of~\cite{Co} that the Hilbert class field~$\mathbb{L}$ of~$\mathbb{K}$ is defined as the maximal abelian unramified extension of~$\mathbb{K}$ and we may assume $\mathbb{L}\subseteq \C$ by fixing an embedding. $\mathbb{L}$ has the property that every ideal~$\mathcal{A} \subseteq \oh_{\mathbb{K}}$ yields a principal ideal~$\mathcal{A} \oh_{\mathbb{L}} \subseteq \oh_{\mathbb{L}}$. More precisely, $\mathbb{L} \supseteq \mathbb{K}$ is a Galois extension with Galois group canonically isomorphic to~$\Ccal\ell_{\mathbb{K}}$. The~$n$-torsion subgroup of~$\Ccal\ell_{\mathbb{K}}$ via the Galois correspondence gives rise to an intermediate extension 
\[
\mathbb{L} \supseteq \mathbb{L}[n] \supseteq \mathbb{K}
\]
with the property that an ideal~$\mathcal{A} \subseteq \oh_{\mathbb{K}}$ yields a principal ideal~$\mathcal{A} \oh_{\mathbb{L}[n]} \subseteq \oh_{\mathbb{L}[n]}$ if and only $[\mathcal{A}] \in \Ccal\ell_{\mathbb{K}}[n]$. This can be inferred from, for instance, the unique factorization property of ideals in~$\oh_{\mathbb{K}}$ and the splitting behavior over~$\mathbb{L}$ of prime ideals in~$\oh_{\mathbb{K}}$, which is outlined in~\cite{Co}. 

Choose an ideal~$\mathcal{A} \subseteq \oh_{\mathbb{K}}$ satisfying~$[ \mathcal{A} ] \in \Ccal\ell_{\mathbb{K}}[n]$ and let $A \in \oh_{\mathbb{K}}^{n \times n}$ be a matrix with elementary divisors~$\mathcal{A},\ldots,\mathcal{A}$ as in the proof of Theorem~1. Then there is an element $u' \in \oh_{\mathbb{L}[n]}$ such that $\mathcal{A} \oh_{\mathbb{L}[n]} = u' \oh_{\mathbb{L}[n]}$. Set $v = \det\,A \slash u^{\prime\,n} \in \oh^*_{\mathbb{L}[n]}$. Then we have~$u^n = \det\,A$ with~$u = u' \sqrt[n]{v}$, and therefore~$\frac{1}{u} \left(\begin{smallmatrix} A & 0 \\ 0 & D \end{smallmatrix}\right) \in \Delta_{n,\mathbb{K}}^\ast$ for~$D$, again, as in the proof of Theorem~1. In particular, we have found a preimage of~$[\mathcal{A}] \in \Ccal\ell_{\mathbb{K}}[n]$ under the homomorphism in Theorem~1~b), whose entries are contained in~$\mathbb{L}[n](\sqrt[n]{v})$.

To finish the proof, we have to show that~$\mathbb{L}[n](\sqrt[n]{v}) \supseteq \Q$ is unramified outside of~$n d_{\mathbb{K}}$. As~$\mathbb{L} \supseteq \mathbb{K}$ is unramified, so is~$\mathbb{L}[n] \supseteq \mathbb{K}$. If $\mathfrak{p} \subseteq \oh_{\mathbb{K}}$ is a prime ideal that does not divide~$n$, then the polynomial~$X^n - v$ is separable modulo~$\mathfrak{p}$, since its derivative~$n X^{n-1}$ does not vanish modulo~$\mathfrak{p}$ and neither does~$v$.
\end{proof}

We turn to the computation of representatives of~$\Delta_{n,\mathbb{K}}^\ast \slash \mu \Gamma_n(\oh_{\mathbb{K}})$, which we have implemented based on the computer algebra package Hecke~\cite{FiHaHoJo}. By virtue of the proof of both Theorem~1 and~\ref{Theorem4}, it is naturally reduced to finding matrices~$\frac{1}{u} A$ of determinant~$1$ with~$A \in \oh_{\mathbb{K}}^{n \times n}$ and elementary divisors~$\mathcal{A},\ldots,\mathcal{A}$ for~$[\mathcal{A}] \in \Ccal\ell_{\mathbb{K}}[n]$. The computation of such~$A$ is achieved by the function \texttt{hermitian extension} in \cite{Rm3}. 

The two key aspects of \texttt{hermitian extension} are the use of the Hilbert class field in \cite{Rm3} and the function \texttt{elementary divisor matrix}. The latter produces an~$n \times n$ matrix~$A$ over~$\oh_{\mathbb{K}}$ with elementary divisors~$\mathcal{A},\ldots,\mathcal{A}$. To this end, we employ pseudo-matrices from Definition~1.4.5 of~\cite{Co}. Specifically, the pseudo-matrix~$\widetilde{A} = (I, (\mathcal{A},\ldots,\mathcal{A}))$ trivially has the desired elementary divisors. The Steinitz form of~$\widetilde{A}$ is a pseudo-matrix~$\widetilde{A}'$, whose associated ideals are trivial, save the last one, which equals the Steinitz class of the ideals associated with~$\widetilde{A}$ (cf.~Theorem~1.2.12 of~\cite{Co}). Since~$[\mathcal{A}] \in \Ccal\ell_{\mathbb{K}}[n]$ by assumption, this Steinitz class is principal and yields the desired matrix~$A$.

We conclude this paper with a brief excerpt of discriminants of the fields generated by the entries of~$\Delta_{n,\mathbb{K}}^\ast$, which we have obtained through our script. Observe that the powers of the discriminant of~$\mathbb{K}$ originate in the Hilbert class field.

\begin{table}[h]
\caption{Discriminants~$d_{\widehat{\K}_n}$ of the fields generated by the entries of~$\Delta_{n,\mathbb{K}}^\ast$ given $n$ and $\mathbb{K} = \mathbb{Q}(\sqrt{-m})$}
\begin{center}
\footnotesize
\begin{tabular}{rccccl@{\hspace{.3em}}ccccc}
\toprule
 $m$  & $-23$                   & $-31$         & $-59$         & $-83$ && $-39$ & $-55$ & $-56$ & $-68$ & $-84$  \\
\cmidrule{2-5} \cmidrule{7-11}
$n$                                  & \multicolumn{4}{c}{3} &&  \multicolumn{5}{c}{4}\\
\midrule
$d_{\widehat{\K}_n}$                   & $3^{14}\,23^3$           & $3^{14}\,31^3$ & $3^{14}\,59^3$ & $3^{14}\,83^3$ && 
 $2^{72}\,3^{16}\,13^{16}$ & $2^{72}\,5^{16}\,11^{16}$ &       $2^{44}\,7^8$ & $2^{48}\,17^8$ & $2^{32}\,3^8\,7^8$  \\
\bottomrule
\end{tabular}
\normalsize
\end{center}
\end{table}


\bigskip
\setlength{\parindent}{0pt}
{\small \textbf{Acknowledgement.} The authors thank Joana Rodriguez for the direct proof of Theorem 2 and Claus Fieker for his help with Hecke~\cite{FiHaHoJo}.} 



\bibliographystyle{plain}
\renewcommand{\refname}{Bibliography}
\bibliography{bibliography_krieg_2021} 

\end{document}